\documentclass{amsart}
\usepackage[latin1]{inputenc}
\usepackage{a4}
\usepackage{morefloats}
\usepackage[english]{babel}
\usepackage{amsmath, amsfonts, amssymb, amsthm}
\usepackage{floatflt}
\usepackage{multicol}
\usepackage{caption}
\usepackage{xcolor}
\usepackage{pdfpages} 
\usepackage{graphicx}

\newtheorem{lem}{Lemma}

\newtheorem*{rem}{Remark}
\newtheorem{pro}{Proposition}
\newtheorem{con}{Construction}
\newtheorem*{ac}{Acknowledgement}

\begin{document}


\title[Some Constructions of the Golden Ratio in an Arbitrary Triangle]{Some Constructions of the Golden Ratio in an Arbitrary Triangle}

\author{Tran Quang Hung}
\address{High school for Gifted students, Hanoi University of Science, Hanoi National University, Hanoi, Vietnam.}
\email{analgeomatica@gmail.com}

\date{\today}

\begin{abstract}
We establish some new constructions of the golden ratio in an arbitrary triangle using symmedians and nine-point circle.
\end{abstract}

\maketitle

\section{Introduction}

The golden ratio often appear in regular polygons \cite{oai,3,6,7,hung} and less in the isosceles triangle \cite{2,8}. The author introduced a construction of the golden ratio in an arbitrary triangle with two symmedians in \cite{9}. Continuing with this idea, we shall introduce some new other constructions of the golden ratio in an arbitrary triangle with symmedians and nine-point circle. Some constructions can be considered as the generalization of classical construction of George Odom in \cite{5}.

\section{The lemmas}

We recall the lemmas with proofs from \cite{9}.

\begin{lem}\label{lem1}
Given a convex cyclic quadrilateral $ABCD$. Two diagonals $AC$ and $BD$ interesect at $P$. Then
$$
\frac{PA}{PC}=\frac{AB\cdot AD}{CB\cdot CD}.
$$
\end{lem}
\begin{figure}[ht!]
\begin{center}
\includegraphics[scale=0.7]{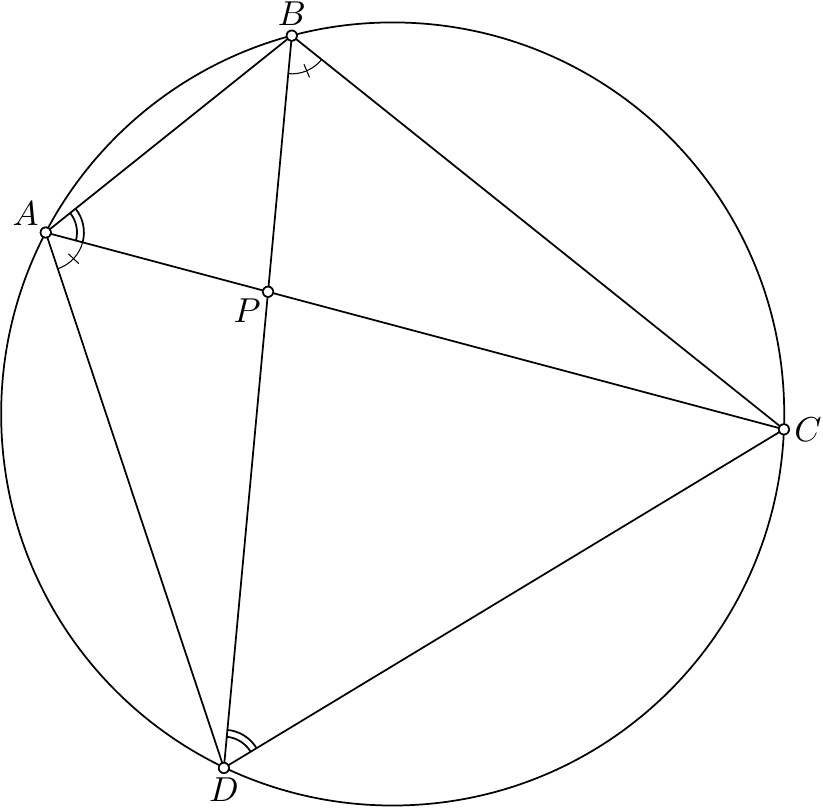}
\end{center}
\caption{Proof of Lemma~\ref{lem1}}
\end{figure}
\begin{proof}
 From the equality of the corresponding angles in the cyclic quadrilateral, we have the similar triangles $\triangle PAB\sim\triangle PDC$ and $\triangle PAD\sim\triangle PBC$. From this, we get the ratios
\begin{equation}\label{eq1.1}
\frac{PA}{PB}=\frac{AD}{BC}
\end{equation}
and
\begin{equation}\label{eq1.2}
\frac{PB}{PC}=\frac{AB}{CD}.
\end{equation}
From \eqref{eq1.1} and \eqref{eq1.2}, we obtain 
$$
\frac{PA}{PC}=\frac{AB\cdot AD}{CB\cdot CD}.
$$
The proof is complete.
\end{proof}

\begin{lem}[Ptolemy's theorem \cite{1}]\label{lem2}
For a cyclic quadrilateral, the sum of the products of the two pairs of opposite sides equals the product of the diagonals.
\end{lem}

This is the famous theorem of plane geometry that can be found in \cite{1}.

Using the concept of homogeneous barycentric coordinates \cite{10}, we give and prove the following lemma.

\begin{lem}\label{lem3}
Let $ABC$ be a triangle inscribed in a circle $(\omega)$. $P$ is a point inside triangle $ABC$.  $P$ has homogeneous barycentric coordinates $(x:y:z)$. $DEF$ is a cevian triangle of $P$. Ray $EF$ meets $(\omega)$ at $Q$. Then
$$
\frac{CA}{yQB}=\frac{BC}{xQA}+\frac{AB}{zQC}.
$$
\end{lem}
\begin{figure}[ht!]
\begin{center}
\includegraphics[scale=0.7]{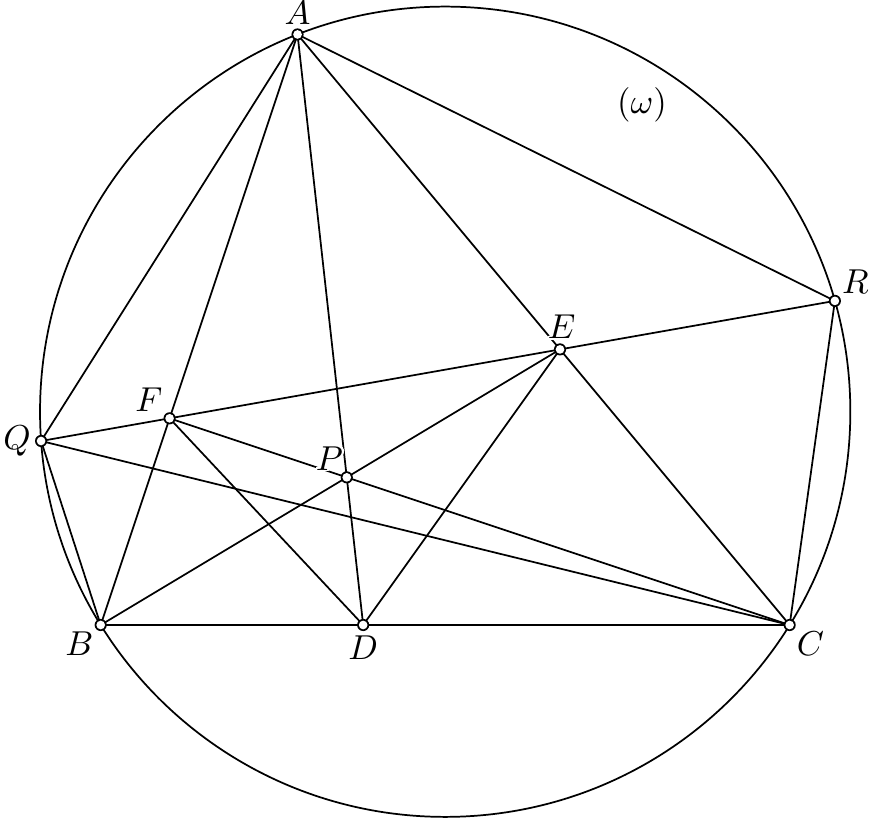}
\end{center}
\caption{Proof of Lemma~\ref{lem3}}
\end{figure}
\begin{proof}If $P$ has homogeneous barycentric coordinates $(x:y:z)$, we get that the barycentric coordinates of $E$, $F$ are $(x:0:z)$ and $(x:y:0)$ respectively. Thus we have the ratio
\begin{equation}\label{eq2.1}
\frac{EA}{EC}=\frac{z}{x}.
\end{equation}
Let the ray $QE$ meet $(\omega)$ second time at $R$. From Lemma~\ref{lem1}, we have
\begin{equation}\label{eq2.2}
\frac{EA}{EC}=\frac{AQ\cdot AR}{CQ\cdot CR}.
\end{equation}
From \eqref{eq2.1} and \eqref{eq2.2}, we deduce
$$
\frac{AQ\cdot AR}{CQ\cdot CR}=\frac{z}{x}.
$$
Thus 
\begin{equation}\label{eq2.3}
zQC=x\frac{AQ\cdot AR}{CR}.
\end{equation}
Similarly, we have the identity
$$
\frac{AQ\cdot AR}{BQ\cdot BR}=\frac{y}{x}.
$$
Thus 
\begin{equation}\label{eq2.4}
yQB=x\frac{AQ\cdot AR}{BR}.
\end{equation}
Using \eqref{eq2.3}, \eqref{eq2.4}, and Lemma~\ref{lem2}, we consider the expression
\begin{align*}
\frac{CA}{yQB}-\frac{AB}{zQC}
&=\frac{CA}{x\frac{AQ\cdot AR}{BR}}-\frac{AB}{x\frac{AQ\cdot AR}{CR}}\\
&=\frac{CA\cdot RB-AB\cdot RC}{x{AQ\cdot AR}}\\
&=\frac{BC\cdot RA}{x{AQ\cdot AR}}\\
&=\frac{BC}{xQA}.
\end{align*}
Therefore,
$$
\frac{CA}{yQB}=\frac{BC}{xQA}+\frac{AB}{zQC}.
$$
This completes the proof of Lemma~\ref{lem3}.
\end{proof}

\section{Constructions and proofs}

\begin{con}\label{con1}
Given a triangle $ABC$ with the circumcircle $(\omega)$.
\begin{itemize}
\item[i)] Constructing the symmedians $BE$ and $CF$.

\item[ii)] Ray $EF$ meets $(\omega)$ at $P$.

\item[iii)] The parallel line from $P$ to $BC$ meets $AB$ and $AC$ at $Q$ and $R$ respectively and meets $(\omega)$ again at $S$.
\end{itemize}
\end{con}
\begin{pro}\label{pro1}
$R$ divides $QS$ into the golden ratio.
\begin{figure}[ht!]
\begin{center}
\includegraphics[scale=0.7]{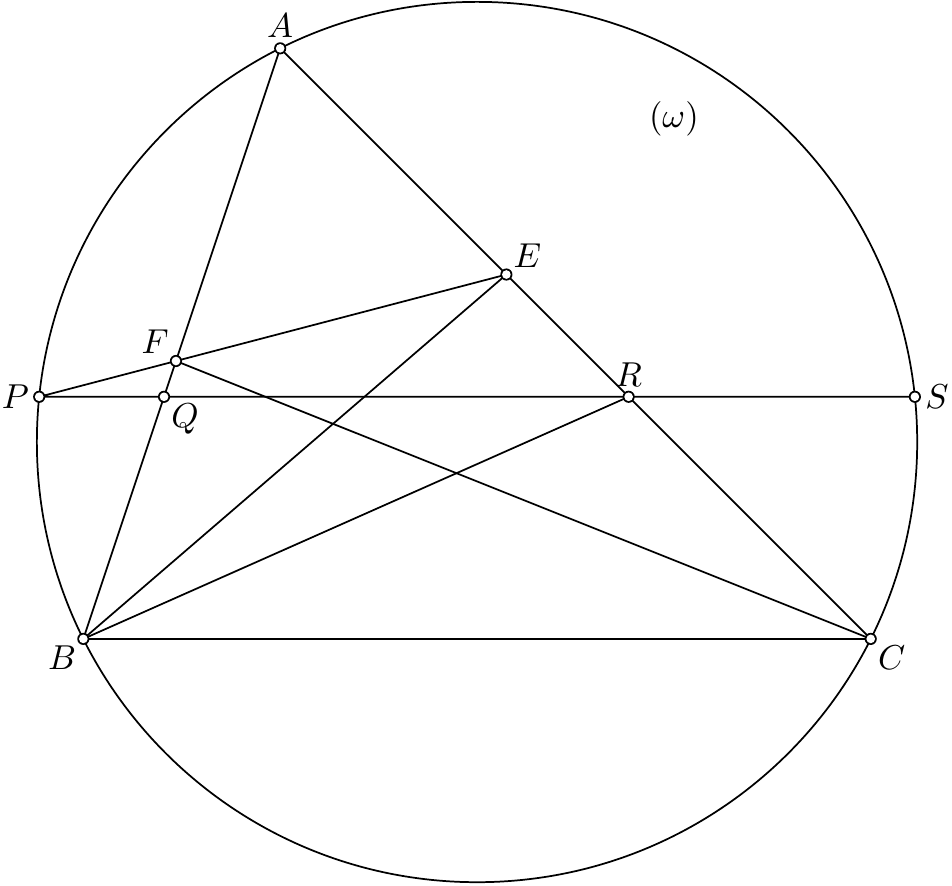}
\end{center}
\caption{The Construction~\ref{con1}}
\end{figure}
\end{pro}
\begin{proof}
It follows from $\triangle ARP\sim\triangle SRC$ and $\triangle AQP\sim\triangle SQB$ that
$$
\frac{RA}{RS}=\frac{AP}{SC}
$$
and
$$
\frac{QS}{QA}=\frac{SB}{AP}.
$$
Therefore,
\begin{equation}\label{eq3.1}
\frac{QS}{RS}\cdot\frac{RA}{QA}=\frac{SB}{SC}.
\end{equation}
Because $QR\parallel BC$,
\begin{equation}\label{eq3.2}
\frac{RA}{QA}=\frac{AC}{AB}.
\end{equation}
Also by $PS\parallel BC$, we have $PSCB$ is an isosceles trapezoid. Hence,
\begin{equation}\label{eq3.3}
SB=PC,\quad SC=PB.
\end{equation}
From \eqref{eq3.1}, \eqref{eq3.2} and \eqref{eq3.3}, we deduce
\begin{equation}\label{eq3.4}
\frac{QS}{RS}=\frac{AB\cdot PC}{CA\cdot PB}.
\end{equation}
\begin{figure}[ht!]
\centering
\includegraphics[scale=0.7]{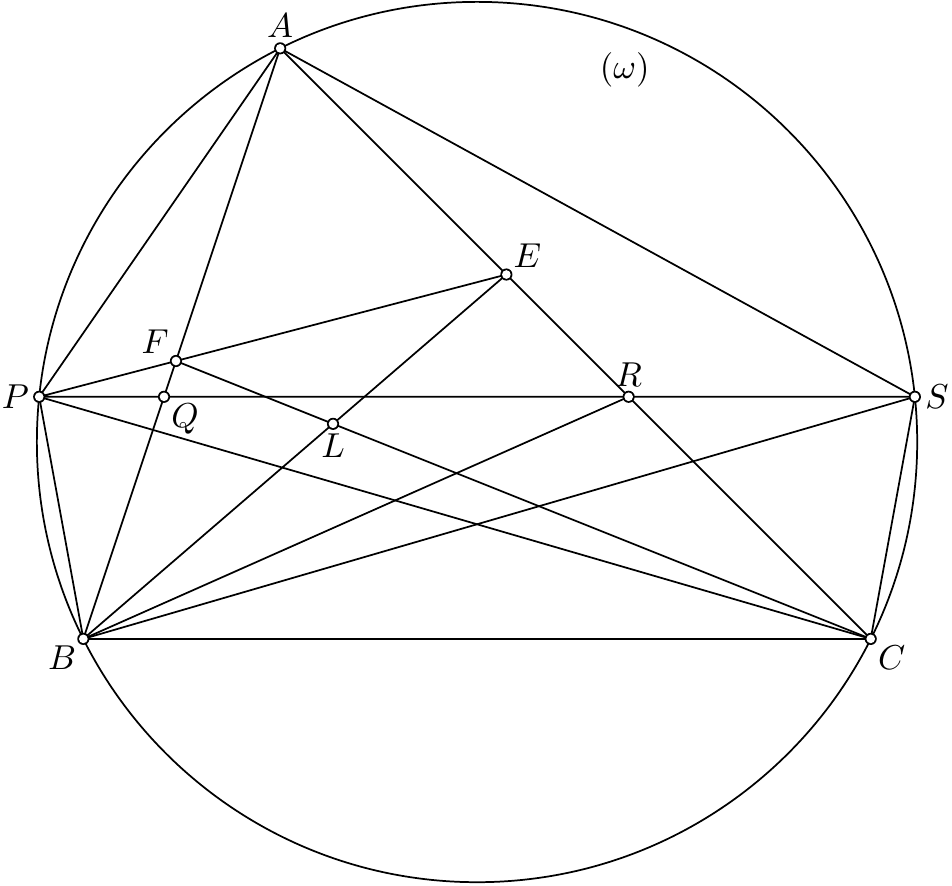}
\caption{Proof of Proposition~\ref{pro1}}
\end{figure}
Note that, by Lemma~\ref{lem2},
\begin{equation}\label{eq3.5}
BC\cdot PA=AB\cdot PC-CA\cdot PB.
\end{equation}
From \eqref{eq3.4} and \eqref{eq3.5}, we infer that
\begin{equation}\label{eq3.6}
\frac{RQ}{RS}=\frac{QS}{RS}-1=\frac{AB\cdot PC}{CA\cdot PB}-1=\frac{AB\cdot PC-AC\cdot PB}{AC\cdot PB}=\frac{BC\cdot PA}{CA\cdot PB}.
\end{equation}
Let $BE$ meet $CF$ at $L$, then $L$ has barycentric coordinates $L(BC^2:CA^2:AB^2)$, see \cite{10}. Apply Lemma~\ref{lem3} for triangle $ABC$ with $L$ and ray $EF$ meet $(\omega)$ at $P$, we obtain
$$
\frac{CA}{CA^2\cdot PB}=\frac{BC}{BC^2\cdot PA}+\frac{AB^2}{AB\cdot PC}.
$$
This is equivalent to
\begin{equation}\label{eq3.7}
\frac{1}{CA\cdot PB}=\frac{1}{BC\cdot PA}+\frac{1}{AB\cdot PC}.
\end{equation}
From \eqref{eq3.5} and \eqref{eq3.7}, we have
$$
\frac{1}{CA\cdot PB}=\frac{1}{BC\cdot PA}+\frac{1}{CA\cdot PB+BC\cdot PA}.
$$
This means
\begin{equation}\label{eq3.8}
\frac{BC\cdot PA}{CA\cdot PB}-\frac{CA\cdot PB}{BC\cdot PA}=1.
\end{equation}
From \eqref{eq3.6} and \eqref{eq3.8}, we secure
$$\frac{RQ}{RS}-\frac{RS}{RQ}=1.
$$
This is enough to show that the ratio
$$
\frac{RQ}{RS}=\frac{\sqrt{5}+1}{2}
$$
which is such the golden ratio. This completes our proof.
\end{proof}

\begin{rem}
This construction can be considered as the generalization of the construction \cite{2} and \cite{5}. If $ABC$ is isosceles at $A$, we obtain the construction in \cite{2}. If $ABC$ is equilateral, we obtain the construction of George Odom in \cite{5}.
\end{rem}

\begin{con}\label{con2}
Given a triangle $ABC$ with the circumcircle $(\omega)$.
\begin{itemize}
\item[i)] Constructing the symmedians $AD$ and $CF$.

\item[ii)] Ray $DF$ meets $(\omega)$ at $P$.

\item[iii)] The parallel line from $P$ to $BC$ meets $AB$ and $AC$ at $Q$ and $R$ respectively and meets $(\omega)$ again at $S$.
\end{itemize}
\end{con}
\begin{pro}\label{pro2}
$R$ divides $SQ$ into the golden ratio.
\begin{figure}[ht!]
\centering
\includegraphics[scale=0.7]{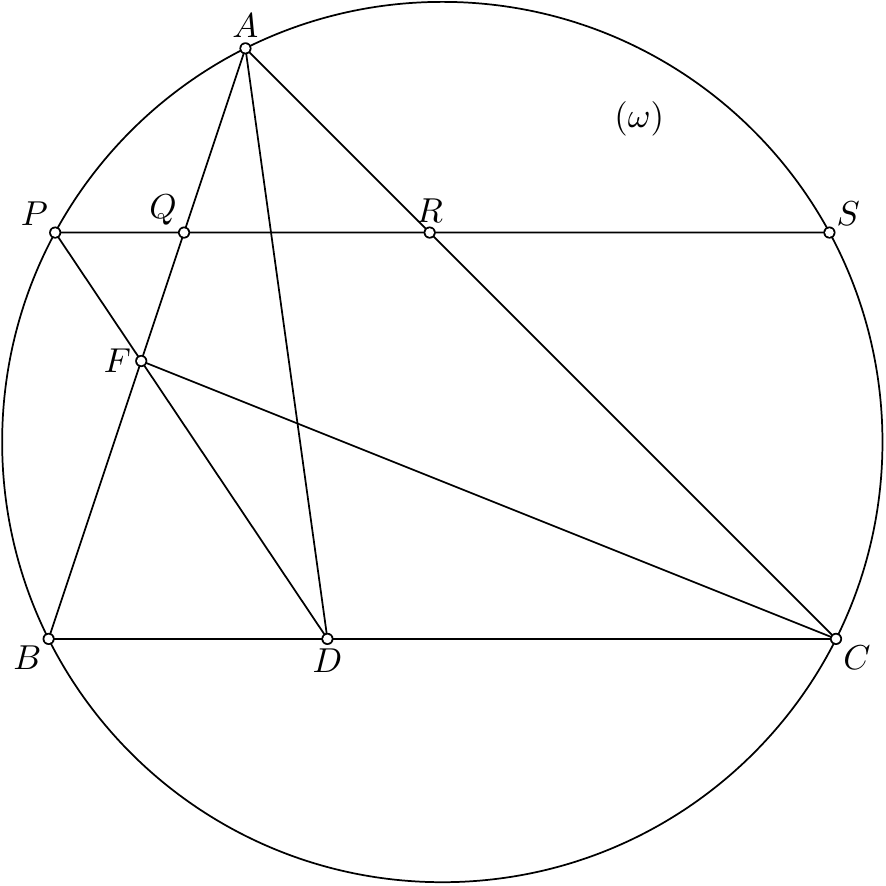}
\caption{The Construction~\ref{pro2}}
\end{figure}
\end{pro}
\begin{proof}
The proof is similar to the proof of Proposition~\ref{pro1}
$$
\frac{RQ}{RS}=\frac{BC\cdot PA}{CA\cdot PB}
$$
or
\begin{equation}\label{eq3.2.1}
\frac{RS}{RQ}=\frac{CA\cdot PB}{BC\cdot PA}.
\end{equation}
Let the symmedians $AD$ and $CF$ meet at $L$, then $L(BC^2:CA^2:AB^2)$, see \cite{1}. Apply Lemma~\ref{lem2} for triangle $ABC$ with $L$ and ray $DF$ meet $(\omega)$ at $P$, we have
$$
\frac{BC}{BC^2\cdot PA}=\frac{AB}{AB^2\cdot PC}+\frac{CA}{CA^2\cdot PB}.
$$
\begin{figure}[ht!]
\centering
\includegraphics[scale=0.7]{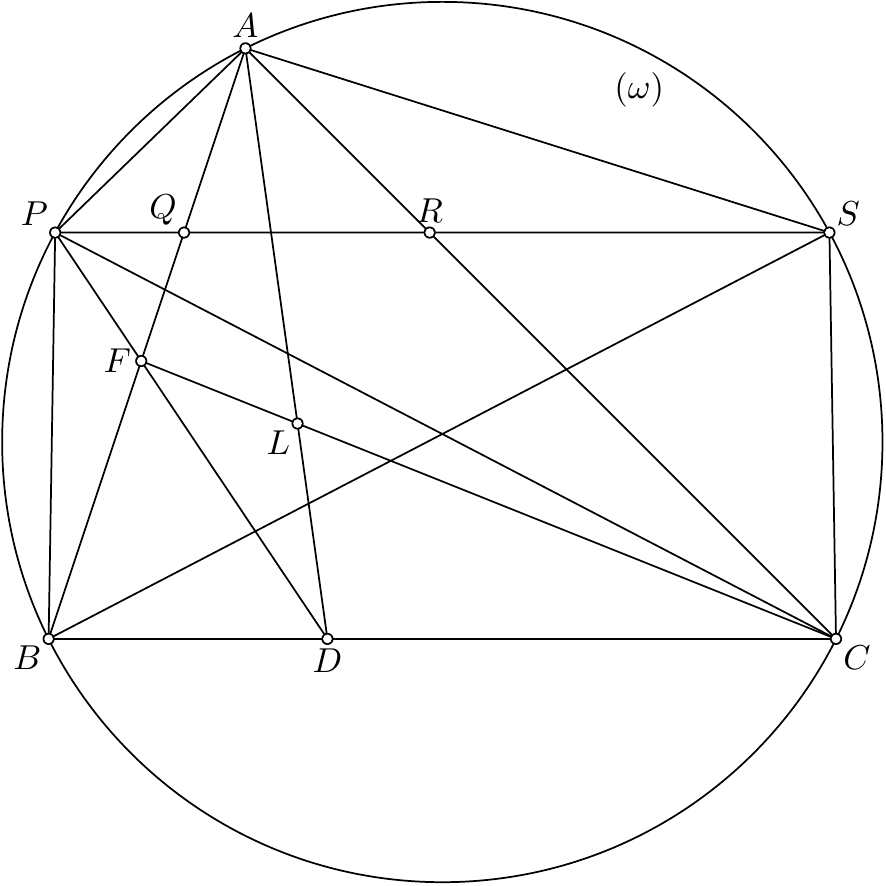}
\caption{Proof of Proposition~\ref{pro2}}
\end{figure}
This is equivalent to
\begin{equation}\label{eq3.2.2}
\frac{1}{BC\cdot PA}=\frac{1}{AB\cdot PC}+\frac{1}{CA\cdot PB}.
\end{equation}
Note that, Lemma~\ref{lem2} then
\begin{equation}\label{eq3.2.3}
AB\cdot PC=BC\cdot PA+CA\cdot PB.
\end{equation}
From \eqref{eq3.2.2} and \eqref{eq3.2.3} we have
\begin{equation}\label{eq3.2.4}
\frac{1}{BC\cdot PA}=\frac{1}{BC\cdot PA+CA\cdot PB}+\frac{1}{CA\cdot PB}.
\end{equation}
This means
\begin{equation}\label{eq3.2.5}
\frac{CA\cdot PB}{BC\cdot PA}-\frac{BC\cdot PA}{CA\cdot PB}=1.
\end{equation}
From \eqref{eq3.2.1} and \eqref{eq3.2.5}, we deduce that 
$$
\frac{RS}{RQ}-\frac{RQ}{RS}=1,
$$
this is enough for deriving equality
$$
\frac{RS}{RQ}=\frac{\sqrt{5}+1}{2}
$$
which is such the golden ratio. This completes our proof.
\end{proof}

\begin{con}\label{con3}
Given a triangle $ABC$ with the nine-point center $N$. $D$, $E$, and $F$ are the midpoints of $BC$, $CA$, and $AB$, respectively. Consider the circle $(\omega)$ with diameter $BC$.
\begin{itemize}
\item[i)] The perpendicular line from $D$ to $ND$ meets circle $(\omega)$ at $S$ and $T$.

\item[ii)] Consider the circle $(\Omega)$ with center $N$ passing through $S$ and $T$.

\item[iii)] Ray $EF$ meets circle $(\Omega)$ at $G$. 
\end{itemize}
\end{con}
\begin{pro}\label{pro3}
$E$ divides $FG$ into the golden ratio.
\begin{figure}[ht!]
\centering
\includegraphics[scale=0.7]{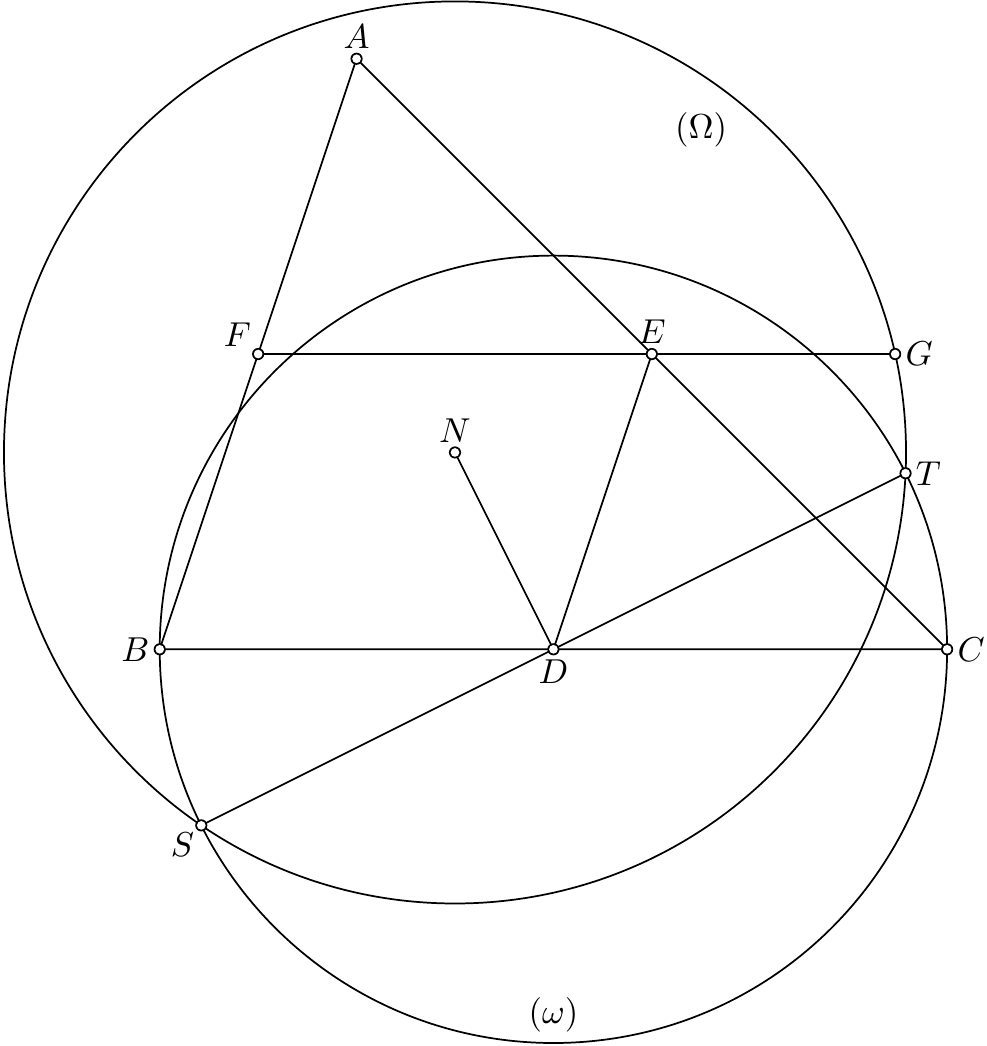}
\caption{The Construction~\ref{con3}}
\end{figure}
\end{pro}
\begin{proof}
We shall show that 
$$
EG\cdot GF=EF^2.
$$
Because $N$ is the nine-point center of triangle $ABC,$ thus $N$ must be the circumcenter of triangle $DEF$. Also, $DEFB$ is a parallelogram, so $EF=BD=DT$.
\begin{figure}[ht!]
\centering
\includegraphics[scale=0.7]{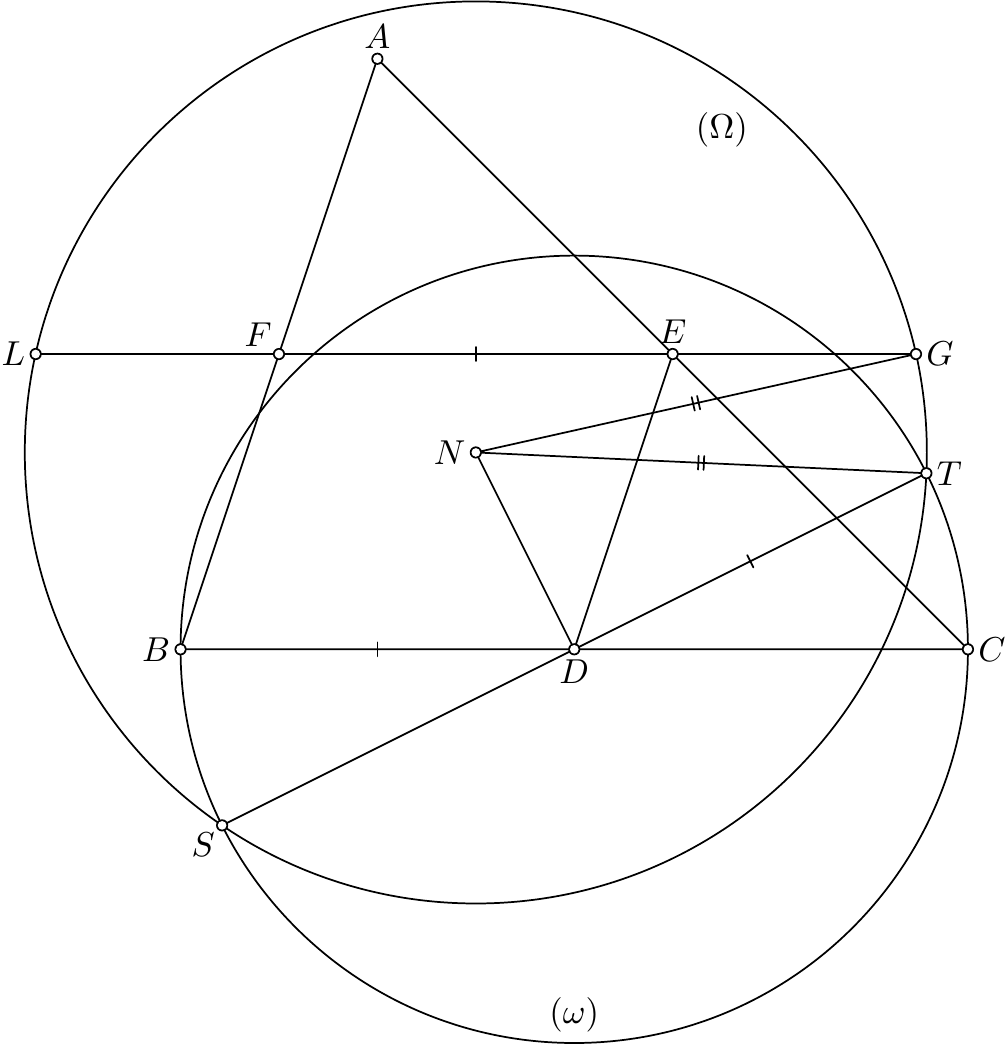}
\caption{Proof of Proposition~\ref{pro3}}
\end{figure}
Let $EF$ meet $(\Omega)$ again at $L$. By symmetry and the intersecting chords theorem,
$$
EG\cdot GF=EG\cdot EL=NG^2-NE^2=NT^2-ND^2=DT^2=DB^2=EF^2.
$$
This proves that $E$ divides $FG$ in the golden ratio. Similarly $F$ divides $EL$ into the golden ratio. We complete the proof.
\end{proof}

\begin{rem}
This construction also can be considered as another generalization of the classical construction of George Odom in \cite{5}. If $ABC$ is an equilateral triangle, we obtain the construction in \cite{5}.
\end{rem}

\begin{con}\label{con4}
Given a triangle $ABC$.
\begin{itemize}
\item[i)] Consider the symmedian $BE$.

\item[ii)] Let $F$ be a point on segment $AE$ such that $\frac{FE}{EC}=\frac{1}{5}$.

\item[iii)] The parallel line from $F$ to $BE$ meets $AB$ at $G$.

\item[iv)] The perpendicular bisectors of $AG$ and $BC$ meet at $K$.

\item[v)] The circle with center $K$ passing though $A$ meets $BC$ at $MN$. 
\end{itemize}
\end{con}
\begin{pro}\label{pro4}$M$ divides $NB$ into the golden ratio.
\begin{figure}[ht!]
\centering
\includegraphics[scale=0.7]{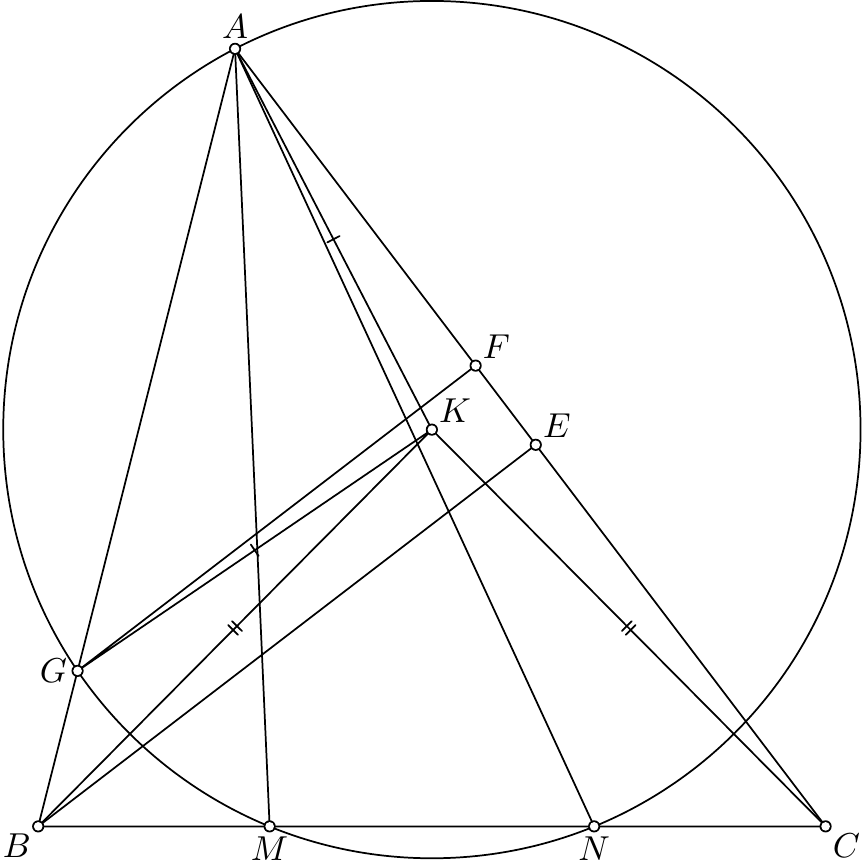}
\caption{The Construction~\ref{con4} and proof of Proposition~\ref{pro4}}
\end{figure}
\end{pro}
\begin{proof}
Because $BE$ is symmedian of $ABC$, we conclude that $\frac{EA}{EC}=\frac{BA^2}{BC^2}$. $EC=5\cdot EF$ or in other words
\begin{equation}\label{eq3.4.1}
\frac{EA}{EF}=5\cdot \frac{BA^2}{BC^2}.
\end{equation}
$K$ lies on perpendicular bisector of $AG$, circle $(K)$ passes through $G$, we have 
\begin{equation}\label{eq3.4.2}
BM\cdot BN=BG\cdot BA=\frac{BG}{BA}\cdot BA^2=\frac{EF}{EA}\cdot BA^2.
\end{equation}
From \eqref{eq3.4.1} and \eqref{eq3.4.2}, we obtain
\begin{equation}\label{eq3.4.3}
BM\cdot BN=\frac{BC^2}{5}.
\end{equation}
Because $K$ lies on perpendicular bisector of $BC$, so $BM=CN$ and we deduce that
\begin{equation}\label{eq3.4.4}
BM+BN=CN+BN=BC.
\end{equation}
From \eqref{eq3.4.3} and \eqref{eq3.4.4}, we get
$$
(BM+BN)^2=5\cdot BM\cdot BN,
$$
this is equivalent to
$$
BN^2-3\cdot BN\cdot BM+BM^2=0
$$
or
$$
\frac{BN}{BM}=\frac{3+\sqrt{5}}{2}
$$
or
$$
\frac{MN}{MB}=\frac{\sqrt{5}+1}{2}
$$
which is such the golden ratio. This completes our proof.
\end{proof}
\begin{rem}
This construction can be considered as generalization of the construction of Dao Thanh Oai in \cite{oai}. If $ABC$ is equilateral triangle or right isosceles, we obtain the construction in \cite{oai}.
\end{rem}

\begin{ac}The author is grateful \textbf{Alexander Skutin} for his proofreading and his suggestion of the nice projective viewpoint for the proofs of Proposition~\ref{pro1} and Proposition~\ref{pro2}.
\end{ac}

\bibliographystyle{plain}

\begin{thebibliography}{00}

\bibitem{1} 
E.W.\ Weisstein, \emph{Ptolemy's Theorem}, from \emph{MathWorld--A Wolfram Web Resource}, {\tt http://mathworld.wolfram.com/PtolemysTheorem.html}.

\bibitem{2} 
T.O.\ Dao, Q.D.\ Ngo, and P.\ Yiu, \emph{Golden sections in an isosceles triangle and its circumcircle}, Global Journal of Advanced Research on Classical and Modern Geometries, \textbf{5} (2016) 93--97.

\bibitem{oai} T.O.\ Dao, \emph{Some golden sections in the equilateral and right isosceles triangles}, Forum Geom., \textbf{16} (2016) 269--272. 

\bibitem{3} D.\ Paunic and P.\ Yiu, \emph{Regular polygons and the golden section}, Forum Geom., \textbf{16} (2016) 273--281.

\bibitem{4} K.\ Hofstetter, \emph{A simple construction of the golden section}, Forum Geom., \textbf{2} (2002) 65--66.

\bibitem{5} G.\ Odom and J.\ van de Craats, \emph{Elementary Problem 3007}, Amer.\ Math.\ Monthly, \textbf{90} (1983) 482; solution, \textbf{93} (1986) 572.

\bibitem{6} M.\ Pietsch, \emph{The golden ratio and regular polygons}, Forum Geom., \textbf{17} (2017) 17--19.

\bibitem{10} P.\ Yiu, \emph{Introduction to the Geometry of the Triangle}, Florida Atlantic University Lecture Notes, 2001; with corrections, 2013, available at
{\tt http://math.fau.edu/Yiu/Geometry.html}.

\bibitem{7} Q.H.\ Tran, \emph{The golden section in the inscribed square of an isosceles right triangle}, Forum Geom., \textbf{15} (2015) 91--92.

\bibitem{hung} Q.H.\ Tran, \emph{Another simple construction of the golden section with equilateral triangles}, Forum Geom., \textbf{17} (2017) 47--48. 

\bibitem{8} Q.H.\ Tran, \emph{Another simple construction of the golden ratio in an isosceles triangle}, Forum Geom., \textbf{17} (2017) 287--288.

\bibitem{9} Q.H.\ Tran, \emph{A Construction of the Golden Ratio in an Arbitrary Triangle}, Forum Geom., \textbf{18} (2018) 239--244.
\end{thebibliography}


\end{document}